\documentclass[12pt, reqno]{amsart}

\usepackage{amsfonts,amssymb,latexsym,amsmath, amsxtra}
\usepackage{verbatim}

\usepackage[OT2,T1]{fontenc}
\DeclareSymbolFont{cyrletters}{OT2}{wncyr}{m}{n}
\DeclareMathSymbol{\Sha}{\mathalpha}{cyrletters}{"58}

\numberwithin{figure}{section}

\newtheorem*{problem}{Problem}

\newtheorem*{question}{Question}
\newtheorem*{defn}{Definition}
\newtheorem{theorem}{Theorem}[section]
\newtheorem*{theorem*}{Theorem}

\newtheorem{conjecture}[theorem]{Conjecture}
\newtheorem{lemma}[theorem]{Lemma}

\newtheorem{proposition}[theorem]{Proposition}
\newtheorem*{conjecture*}{Conjecture}
\newtheorem*{histdefn}{Historical Definition}
\newtheorem*{moderndefn}{Modern Definition}

\theoremstyle{definition}
\newtheorem{definition}[theorem]{Definition}

\theoremstyle{remark}	
\newtheorem*{remark}{Remark}

\newtheorem*{remark*}{Remark}
\newtheorem*{remarks*}{Remarks}

\numberwithin{equation}{section}

\newcommand{\R}{\mathbb R}
\newcommand{\N}{\mathbb N}
\newcommand{\Z}{\mathbb Z}
\newcommand{\C}{\mathbb C}

\newcommand{\Q}{{\mathbb Q}}

\makeatletter
\newcommand*{\rom}[1]{\expandafter\@slowromancap\romannumeral #1@}
\makeatother

\def\({\left(}
\def\){\right)}

\begin{document}

\title{Radial limits of mock theta functions}
\author{Kathrin Bringmann}
\address{Mathematical Institute\\University of
Cologne\\ Weyertal 86-90 \\ 50931 Cologne \\Germany}
\email{kbringma@math.uni-koeln.de}
\author{Larry Rolen}
\email{lrolen@math.uni-koeln.de}

\thanks{The research of the first author was supported by the Alfried Krupp Prize for Young University Teachers of the Krupp foundation and the research leading to these results has received funding from the European Research Council under the European Union's Seventh Framework Programme (FP/2007-2013) / ERC Grant agreement n. 335220 - AQSER.  The second author thanks the University of Cologne and the DFG for their generous support via the DFG Grant D-72133-G-403-151001011, funded under the Institutional Strategy of the University of Cologne within the German Excellence Initiative.}

\date{\today}

\begin{abstract}
Inspired by the original definition of mock theta functions by Ramanujan, a number of authors have considered the question of explicitly determining their behavior at the cusps. Moreover, these examples have been connected to important objects such as quantum modular forms and ranks and cranks by Folsom, Ono, and Rhoades. Here we solve the general problem of understanding Ramanujan's definition explicitly for any weight $\frac12$ mock theta function, answering a question of Rhoades. Moreover, as a side product, our results give a large, explicit family of modular forms.  
\end{abstract}
\maketitle

\section{Introduction}
In this paper, we study a general problem of Rhoades \cite{RhoadesDefn} on the nature of mock theta functions near the cusps, inspired by important examples of Ramanujan \cite{BerndtRamaLetters}, Folsom, Ono, and Rhoades \cite{FolsomOnoRhoades}, and others (for example, see \cite{Susie,FolsomOnoRhoadesConf,Zudilin}). In particular, we resolve the problem and further prove a related conjecture of Rhoades, in addition to describing large, explicit families of quantum modular forms.

Before stating the precise question of Rhoades, we briefly recall the history of the mock theta functions. They were first introduced by Ramanujan in 1920 in his famous ``deathbed'' letter to Hardy (see pages 220-224 of \cite{BerndtRamaLetters}). Since then, they have provided an enormous number of intriguing questions and their underlying structure remained a deep mystery for decades \cite{Watson}. Thanks to the pioneering work of Zwegers \cite{Zwegers} and Bruinier and Funke \cite{BruinierFunke}, we finally understand this structure in terms of so-called harmonic Maass forms (see Section 2.1 for a precise definition). This revelation has spawned a completely new field of number theory with myriad applications, which range from moonshine \cite{ChengDuncanHarvey}, combinatorics \cite{BringmannOnoAnnals}, black holes \cite{ZagierBlackHoles}, and central derivatives of elliptic curves \cite{BruinierOnoAnnals}, just to name a few (see also the surveys \cite{OnoCDM,ZagierSurveryMock}). 

In this paper, we revisit Ramanujan's original definition in a very explicit manner. His definition is vague and in some sense strange, which makes it difficult to interpret it correctly. However, he offered a list of 17 motivating examples of what he called Eulerian series, meaning that they are $q$-hypergeometric series.
For instance, Ramanujan considered the $q$-series defined by
\[f(q):=\sum_{n\geq0}\frac{q^{n^2}}{(-q)_n^2},\]
where throughout $q:=e^{2\pi i \tau}$ with $\tau\in\mathbb H$ and, for $n\in\N_0\cup\{\infty\}$, $(a)_n:=(a;q)_n:=\prod_{j=0}^{n-1}(1-aq^j).$ 
He noticed that all of his examples look like classical modular forms (or theta functions, as he called them) as one approaches roots of unity. However, he conjectured that they do not arise in a trivial manner from modular forms.  Specifically, his definition is as follows.
\begin{histdefn}[Ramanujan \cite{BerndtRamaLetters}]
A mock theta function is a function $F(q)$, defined for $|q|<1$, satisfying the following conditions:
\begin{enumerate}
\item There are infinitely many roots of unity $\xi$ such that as $q$ approaches $\xi$ radially from inside the unit disk, $F(q)$ grows exponentially. 
\item For every root of unity $\xi$, there exists a (weakly holomorphic) modular form $M_{\xi}(q)$ and a rational number $\alpha$ such that $F(q)-q^{\alpha}M_{\xi}(q)$ is bounded as $q\rightarrow\xi$ radially. 
\item There does not exist a single (weakly holomorphic) modular form $M(q)$ such that $F(q)-M(q)$ is bounded as $q$ approaches any root of unity.
\end{enumerate}
\end{histdefn}

For the mock theta function $f$ defined above, Ramanujan \cite{BerndtRamaLetters} noted that it is bounded as one approaches odd order roots of unity, and he made a specific claim about its behavior near even order roots of unity (where there is a pole). Ramanujan's claim was later proven by Watson \cite{Watson}, and it states that for any primitive even order $2k$ root of unity $\xi$

\begin{equation}\label{Watsonf(q)}\lim_{q\rightarrow\xi}\left(f(q)-(-1)^kb(q)\right)=O(1),\end{equation}
where $b(q):=\frac{(q)_{\infty}}{(-q)_{\infty}^2}$ is a modular form (up to a power of $q$). This shows that $f$ satisfies parts (i) and (ii) of the historical definition. Moreover, the $O(1)$ constants in (\ref{Watsonf(q)}) can be given explicitly for primitive even order $2k$ roots of unity $\xi$ \cite{FolsomOnoRhoades}:
\begin{equation}\label{Quantumf(q)}\lim_{q\rightarrow\xi}\left(f(q)-(-1)^kb(q)\right)=-4\sum_{n=0}^{k-1}\left(-\xi;\xi\right)_n^2\xi^{n+1}.\end{equation}
Inspired by Ramanujan's observation, Folsom, Ono, and Rhoades recently fit (\ref{Quantumf(q)}) into an infinite family (see \cite{FolsomOnoRhoades}, Theorem 1.2). Moreover, it turns out that their results give a deep and surprising connection between the generating functions of the important combinatorial sequences counting ranks, cranks, and unimodal sequences.

Although much progress had been made in understanding the mock theta functions from a modern point of view, it is only recently that experts in the field of Maass forms have turned back to take a closer look at Ramanujan's original ideas. In particular, Berndt remarked in 2013 that no one had actually proven that any of Ramanujan's mock theta functions satisfy his own definition \cite{BerndtMockTheta}. This problem was solved by Griffin, Ono, and the second author \cite{GOR}, who proved that all of Ramanujan's mock theta functions satisfy the historical definition (see \cite{RhoadesDefn} for related results and questions). This progress was made possible by Zwegers' construction \cite{Zwegers} of functions satisfying what we refer to as the modern definition of a mock theta function. 
Roughly speaking, the modern definition, due to Zagier, is as follows (for the exact definition of the related objects, see Section 2.1).

\begin{moderndefn}[Zagier \cite{ZagierSurveryMock}, Section 5]
A mock theta function is a function $F(q)$, defined for $|q|<1$, which can be written as $q^{\alpha}G(q)$ for some $\alpha\in\Q$ and some function $G$ which is the holomorphic part of a harmonic Maass form, such that the shadow of $G$ is a  unary theta function.
\end{moderndefn}

Although the result of Griffin, Ono, and the second author answers Ramanujan's conjecture in an abstract sense, it does not address his question of explicitly finding the functions and constants in part (ii) of the historical definition. It is thus natural to ask: What is the deeper structure underlying relations like (\ref{Quantumf(q)})? Concretely, we consider the following problem to fully address the final challenge of Ramanujan. 
\begin{problem}
For a general mock theta function, how can one determine the modular forms $M_{\xi}$ in Ramanujan's definition in a systematic and explicit way (preferably in terms of theta functions), and how can one compute the constants in part (ii) of Ramanujan's definition as finite sums?
\end{problem}

In order to address this problem in a uniform manner, we first need a convenient set of functions to generate all the mock theta functions. This is provided by the universal mock theta function $g_2$ of Gordon and McIntosh \cite{GordonMcIntosh}. We first recall the definition

\begin{equation}\label{g2defn}g_2(\zeta;q):=\sum_{n\geq0}\frac{(-q)_nq^{\frac{n(n+1)}2}}{(\zeta)_{n+1}(\zeta^{-1}q)_{n+1}}.\end{equation}
It is well-known that $g_2$ is a mock Jacobi form, which implies that if $\zeta:=e^{2\pi i z}$ and $z$ is a torsion point (i.e., $z\in\Q\tau+\Q)$, then  $g_2(\zeta;q)$ is a mock theta function \cite{Kang}. The reason this function is called universal is that all mock theta functions of weight $1/2$ can be expressed in terms of linear combinations of specializations of $g_2$ and classical modular forms (for more on universality and examples, see \cite{GordonMcIntosh}). 
Using the function $g_2$, we now state Rhoades'  question, where throughout we let $e(x):=e^{2\pi i x}$ and $\zeta_k^h:=e(h/k)$.  
\begin{question}[Rhoades \cite{RhoadesDefn}, Question 3.4]
How can one explicitly determine modular forms $f_{a,b,A,B,h,k}$ in a uniform way such that 
$g_2(\zeta_b^aq^A;q^B)-f_{a,b,A,B,h,k}(q)$ is bounded as $q\rightarrow\zeta_k^h$? Moreover, how can one describe finite formulas for the constants
\[Q_{a,b,A,B,h,k}:=\lim_{q\rightarrow\zeta_k^h}\(g_2\left(\zeta_b^aq^A;q^B\right)-f_{a,b,A,B,h,k}(q)\)?\] 
\end{question}

\noindent
{\it Two remarks.}

\noindent 1) Here and throughout we abuse notation and refer to $f_{a,b,A,B,h,k}$ as a modular form if there exist $\alpha,\beta\in\Q$ for which $q^{\alpha}f_{a,b,A,B,h,k}(q^{\beta})$ is modular on some congruence subgroup.

\smallskip
\noindent
2) Although Rhoades asked an analogous question for another universal mock theta function $g_3$, (6.1) of \cite{GordonMcIntosh} shows that $g_3$ may be expressed in terms of $g_2$. Thus, Theorem 1.1 does give a finite, simple answer to Rhoades' question in this case as well; we leave the details to the interested reader. 

\smallskip
\noindent
3) We may, and do, assume throughout without loss of generality that $0\leq\frac hk,\frac AB,\frac ab<1$, and also that $(h,k)=(a,b)=(A,B)=1$.
\smallskip

In this paper, we completely solve the problem of Rhoades.
\begin{theorem}\label{mainthm}
The functions $f_{a,b,A,B,h,k}$ in Rhoades' question may be expressed as simple linear combinations of a finite list of canonical theta functions. Moreover, this decomposition yields simple, finite formulas for the resulting constants $Q_{a,b,A,B,h,k}$.
\end{theorem}
\begin{remark}
We give the explicit forms of the modular forms $f_{a,b,A,B,h,k}$ and the constants $Q_{a,b,A,B,h,k}$ in Propositions 5.1,\, 5.2, and 5.3. We invite the reader to read Section 6 for a simple example which illustrates the theorem. 
\end{remark}

As another application of Theorem \ref{mainthm}, we recall that our question is closely related to a new type of nearly modular object known as a quantum modular form (see Section 2.2 for the definition). Although this is a very nascent field, the known examples of quantum modular forms have already been related to many important topics such as knot invariants \cite{ZagierVass}, combinatorics and partial theta functions \cite{FolsomOnoRhoades}, and Maass waveforms \cite{ZagierQMF}. In particular, quantum modular forms related to mock theta functions were recently considered in \cite{RhoadesLimChoi} 
from a formal point of view. Our result gives an explicit understanding of a very large family of quantum modular forms, which itself subsumes many of the previous examples. 
As another corollary of our explicit procedure, Theorem \ref{mainthm} also allows us to answer a conjecture of Rhoades about the finiteness of explicit procedures for choosing $f_{a,b,A,B,h,k}$. 
\begin{conjecture}[Rhoades, \cite{RhoadesDefn}]\label{FinitenessConj}
For any fixed $a,b,A,B$, as $\frac hk$ ranges over $\Q$, only finitely many modular forms $f_{a,b,A,B,h,k}$ are needed to cancel out the singularities of $g_2\(\zeta_b^aq^A;q^B\).$
\end{conjecture}
The precise statement of Theorem \ref{mainthm} in Propositions 5.1,\, 5.2, and 5.3 shows that this conjecture is true, and that in fact at most 3 modular forms $f_{a,b,A,B,h,k}$ are required in general. 

The paper is organized as follows. In Section 2, we recall some definitions and basic results from the theories of harmonic Maass forms, quantum modular forms, and $q$-hypergeometric series. In particular, we review an important formula of Mortenson which is crucial for the proof of our main theorem. In Section 3, we analyze the location of the zeros and poles of the hypergeometric series defining $g_2(\zeta_b^aq^A;q^B)$. We then give some estimates in Section 4 needed for the proof of Theorem \ref{mainthm}, which itself is proven on a case-by-case basis in Section 5. We conclude in Section 6 with an illuminating example and some concluding discussion. 
\section*{Acknowledgements}
\noindent The authors are grateful to Bruce Berndt, Minjoo Jang, and Steffen L\"obrich for useful comments which improved the paper, as well as to Robert Rhoades for useful discussions related to the paper. 
\section{Preliminaries}

\subsection{Harmonic Maass forms}\label{HMFSubsection}

Maass forms were introduced by Maass \cite{Maass} and generalized by Bruinier and Funke \cite{BruinierFunke} to allow growth at the cusps. 
To give their definition, we first recall the usual weight $\kappa$ hyperbolic Laplacian operator given by ($\tau=u+iv$)
\begin{equation*}
\Delta_{\kappa}:= -v^2\left( \frac{\partial^2}{\partial u^2}+\frac{\partial^2}{\partial v^2} \right) +i\kappa v\left(  \frac{\partial}{\partial u}  +i\frac{\partial}{\partial v} \right).
\end{equation*}
Moreover, for odd $d$, we set
\begin{equation*}
\varepsilon_d:= \left\{\begin{array}{lr}1 &\textrm{ if }d\equiv 1\pmod{4},\\i&\textrm{ if }d\equiv 3\pmod{4},\end{array}\right.
\end{equation*}
and let $(\frac{\cdot}{\cdot})$ be the usual Jacobi symbol. 
 For $\kappa\in\frac12+\Z$, we then define Petersson slash operator for $\gamma=\(\begin{smallmatrix}a&b\\ c&d\end{smallmatrix}\)\in\Gamma_0(4)$ by
\[F\vert_{\kappa}\gamma(\tau):=\varepsilon_d^{2\kappa}\left(\frac cd\right)(c\tau+d)^{-\kappa}F\left(\frac{a\tau+b}{c\tau+d}\right).\] 

We now give the definition of harmonic Maass forms. 
\begin{definition}\label{MaassDefn} A harmonic (weak) Maass form of weight $\kappa\in\frac12+\Z$ on a congruence subgroup $\Gamma\subseteq\Gamma_0(4)$ is any $\mathcal C^2$ function $F\colon \mathbb H\to \C$ satisfying:
\begin{enumerate}
\item For all $\gamma\in \Gamma,$ we have that $F\vert_{\kappa}\gamma  = F.$
\item  We have that $\Delta_{\kappa} (F)=0.$
\item There exists a polynomial $P_F(q)=\sum_{n\leq0} C_F(n)q^n\in \C[q^{-1}]$ such that $F(\tau)-P_F(q)=O(e^{-\varepsilon v})$ for some $\varepsilon >0$ as $v\to +\infty$. We require analogous
conditions at all of the cusps of $\Gamma$.
\end{enumerate}
\end{definition}

We denote the space of weight $\kappa$ harmonic Maass forms on $\Gamma$ by $H_{\kappa}(\Gamma)$.
 A key property of harmonic Maass forms is that they canonically split into a holomorphic piece and a non-holomorphic piece. 
Namely, if we define
$\Gamma(\alpha,v):=\int_v^{\infty}e^{-t}t^{\alpha-1}dt,$
then any $F\in H_{\kappa}(\Gamma)$ (for $\kappa\neq 1$) decomposes as $F=F^++F^-$, where 
 \[F^+(\tau)=\displaystyle\sum_{n\gg-\infty}c_F^+(n)q^n,\quad\quad\quad
F^-(\tau)=\displaystyle\sum_{n<0}c_F^-(n)\Gamma(1-\kappa,4\pi|n|v)q^n,\]
for some complex numbers $c_F^+(n),c_F^-(n)$. 
We refer to $F^+$ as the holomorphic part and to $F^-$ as the non-holomorphic part. We call a $q$-series which is the holomorphic part of some harmonic Maass form a mock modular form. Another crucial operator in the theory of harmonic Maass forms is given by $\xi_{\kappa}:=2iv^{\kappa}\overline{\frac{\partial}{\partial\overline{\tau}}}$, which defines a map $\xi_{\kappa}\colon H_{\kappa}(\Gamma)\rightarrow S_{2-\kappa}(\Gamma)$. Bruinier and Funke proved \cite{BruinierFunke} that $\xi_{\kappa}$ surjects onto $S_{2-\kappa}(\Gamma)$. Following Zagier, we call $\xi_{\kappa} (F)$ the shadow of $F^+$. 

We now review a few facts regarding Zwegers' $\mu$ function, which he used \cite{Zwegers} to write several of Ramanujan's mock theta functions in terms of canonical Appell sums. Specifically, let
 \[\mu(z_1,z_2;\tau):=\frac{a^{\frac12}}{\vartheta(z_2)}\sum_{n\in\Z}\frac{(-b)^nq^{\frac{n^2+n}2}}{1-aq^n}\]
where $a:=e^{2\pi i z_1}$, $b:=e^{2\pi i z_2}$,  and
\begin{equation*}\vartheta(z):=\vartheta(z;\tau)=-iq^{\frac18}\zeta^{-\frac12}(q)_{\infty}(\zeta)_{\infty}\left(\zeta^{-1}q\right)_{\infty}\end{equation*}
is Jacobi's theta function. 
We need the following shifting property of $\mu$, where
\[\Theta\left(z_1, z_2, z\right):=\Theta\left(z_1, z_2, z;\tau\right)=\frac{1}{2\pi i}\frac{\vartheta'(0)\vartheta(z_1+z_2+z)\vartheta(z)}{\vartheta(z_1)\vartheta(z_2)\vartheta(z_1+z)\vartheta(z_2+z)}.\]
\begin{lemma}[\cite{Zwegers}, Proposition 1.4 (7)]\label{mushiftlemma}
For $z_1,z_2,z_1+z,z_2+z\not\in\Z\tau+\Z$, the following identity holds:
\begin{equation*}\mu(z_1+z,z_2+z;\tau)-\mu(z_1,z_2;\tau)=\Theta\left(z_1, z_2, z\right).\end{equation*}
\end{lemma}
The function $\mu$ is what Zwegers calls a mock Jacobi form, and is fundamental to our modern understanding of mock theta functions. In particular, if we specialize $z_1,z_2\in\Q\tau+\Q$, then $\mu(z_1,z_2;\tau)$ is a mock theta function. \subsection{Quantum modular forms} 
We next give the definition of quantum modular forms (see \cite{ZagierQMF} for a general survey). 

\begin{defn}
A function $f\colon \mathbb{Q}\rightarrow\C$ is a quantum modular form of weight $\kappa\in\frac12\Z$ on a congruence subgroup $\Gamma$  if for all $\gamma\in\Gamma$, the cocycle
\[r_{\gamma}(x):=f|_{\kappa}(1-\gamma)(x)\]

\noindent extends to an open subset of\, $\R$ and is analytically ``nice''. Here ``nice'' could mean continuous, smooth, real-analytic etc. 

\end{defn}

One of the most striking examples of a quantum modular form is given by Kontsevich's ``strange'' function $F(q)$ \cite{ZagierVass} 
\[F(q):=\sum_{n\geq0}(q)_n.\]
This function is strange as it does not converge on any open subset of $\C$, but is a finite sum if $q$ is a root of unity. Zagier's study of $F$ depends on the sum of tails identity
\begin{equation}\label{sumoftails}
\displaystyle\sum_{n\geq 0}\left(\eta(\tau)-q^{\frac1{24}}\left(q\right)_n\right)=\eta(\tau)D\left(\tau\right)+\sqrt 6\widetilde{\eta}(\tau),
\end{equation}
\noindent
where $\eta(\tau):=q^{1/24}(q)_{\infty}$, $D(\tau):=-\frac12+\sum_{n\geq1}\frac{q^n}{1-q^n},$ and for any weight $2-\kappa$ cusp form $F(\tau):=\sum_{n\geq 1}a_F(n)q^n$, we define the Eichler integral
$\widetilde{F}(\tau):=\sum_{n\geq1}a_F(n) n^{\kappa-1}q^n.$
The key observation of Zagier is that in (\ref{sumoftails}), the functions $\eta(\tau)$ and $\eta(\tau)D(\tau)$ approach zero to infinite order as $\tau\rightarrow\frac hk$. Hence, at a root of unity $\xi$, $F(\xi)$ is the limiting value of $\widetilde{\eta}$, which he shows has quantum modular properties. In fact, as shown in \cite{BringmannRolenConference}, general Eichler integrals of half-integral weight modular forms are closely connected with the radial limits of mock theta functions. This is related to recent work of Choi, Lim, and Rhoades \cite{RhoadesLimChoi}, which implies the following result about the function $Q\colon\Q\rightarrow\C$ defined by $Q(h/k):=Q_{a,b,A,B,h,k}$.
\begin{theorem}[\cite{RhoadesLimChoi}, Theorem 1.5]
Assuming the notation above, $Q(x)$ is a quantum modular form of weight $1/2$ on some congruence subgroup whose cocycles are all real-analytic on $\R$ except at one point.
\end{theorem}

\subsection{Some useful formulas for $q$-series}
In this subsection, we give several useful $q$-series identities which are crucial to the proof of Theorem \ref{mainthm}. 
Since we know many properties of $\mu$ very explicitly, it is very useful in the proof of Theorem \ref{mainthm} to write $g_2$ in terms of $\mu$ and the function 
\[K(z;\tau):=\frac{\eta(2\tau)^4}{i\zeta\eta(\tau)^2\vartheta(2z;2\tau)},\] as in the following lemma.
\begin{lemma}[\cite{Kang}, Theorem 1.1]\label{Lemma24}
The following identity holds:
\begin{equation*}g_2(\zeta;q)=K(z; \tau)-i q^{-\frac14}\mu(2z,\tau;2\tau)\end{equation*}

\end{lemma}
In order to estimate the growth of certain $q$-series at the cusps, we also need the following important relation between $g_2$ and an Appell sum.
\begin{lemma}[\cite{Kang}, Lemma 3.2]\label{g2otherappelllemma}
The following identity holds:
\begin{equation*}g_3(-\zeta;q)=\frac{1}{(q)_{\infty}}\sum_{n\in\Z}\frac{(-1)^nq^{\frac{3n(n+1)}2}}{1+\zeta q^n}.\end{equation*}
\end{lemma}
Yet another $q$-hypergeometric series identity that we require is a beautiful bilateral series summation of Mortenson. Before stating it, we first let
\begin{equation*}\begin{aligned}T(\zeta;q):=& -\frac{i\eta(2\tau)^4}{\zeta\eta(\tau)^2\vartheta\left(2z;2\tau\right)}-\frac{i\eta(2\tau)^{10}\vartheta\left(2z+\frac12;2\tau\right)}{2\zeta^2q^{\frac14}\eta(\tau)^4\eta(4\tau)^4\vartheta\left(2z;2\tau\right)\vartheta\left(2z+\tau+\frac12;2\tau\right)}\\ &-\frac{i\eta(2\tau)^4\vartheta(z;\tau)}{2q^{\frac14}\zeta^2\eta(4\tau)^2\vartheta\left(z+\frac12;\tau\right)\vartheta\left(2z+\tau+\frac12;2\tau\right)}.\end{aligned}\end{equation*}
Then the identity of Mortenson is given as follows (note that, as pointed out to the authors by Minjoo Jang and Steffen Löbrich, there is a slight typo in the of \cite{Mortenson} concerning the $g_3$ on the right hand side of the identity). 
\begin{lemma}[\cite{Mortenson}, Corollary 5.2]\label{MortensonBilaterallemma}
The following identity holds:
\begin{equation*}\begin{aligned}&g_2(\zeta;q)+\frac12\sum_{n\geq0}\frac{q^n\left(\zeta^{-1}q\right)_n(\zeta)_n}{(-q)_n}=-\frac{i\zeta^{\frac12}\vartheta(z;\tau)}{2q^{\frac1{24}}\eta(2\tau)}g_3(-\zeta;q)+T(\zeta;q).\end{aligned}\end{equation*}
\end{lemma}

\begin{remark}
In simplifying Mortenson's formula for $T(\zeta;q)$ in Corollary 5.2 of \cite{Mortenson}, we have used the identity
\begin{equation}\label{ThetaIdent}\vartheta(\tau;2\tau)=-iq^{-\frac14}\frac{\eta(\tau)^2}{\eta(2\tau)}.\end{equation}
\end{remark}
We return to a closer study of Lemma 2.6 in Section 4 after analyzing the poles of the hypergeometric series defining $g_2$. 
\section{Poles and roots of $g_2(\zeta_b^aq^A;q^B)$}
In order to explicitly find the modular forms to cancel the poles of $g_2(\zeta_b^aq^A;q^B)$, we first need to know where they lie. That is, we need to determine the set
\[\mathcal P_{a,b,A,B}:=\left\{\frac hk\in\Q\colon g_2\left(\zeta_b^aq^A;q^B\right)\text{ has a pole as }q\rightarrow\zeta_k^h\right\}.\] 
The first step is to find the zeros and poles of the hypergeometric series defining $g_2(\zeta_b^aq^A;q^B)$, which we begin in the next subsection. We remark that although the results of the next subsection determine the poles arising from the denominator, convergence at other roots of unity is a subtle question which requires careful analysis.
\subsection{Zeros in the denominator of $g_2$}
The first obvious situation potentially leading to a pole of $g_2(\zeta_b^aq^A;q^B)$ is when the relevant specialization of (\ref{g2defn}) has a pole coming from a zero in the denominator. 
Note that $g_2(\zeta;q)$ has a pole in some term in (\ref{g2defn}) exactly if
$(\zeta)_{\infty}(\zeta^{-1}q)_{\infty}=0,$ which is equivalent to $z\in\Z\tau+\Z$. 
We now specialize the parameters to $\zeta=\zeta_b^a\zeta_k^{hA}$ and $q=\zeta_k^{hB}$. Then the $n$-th term in (\ref{g2defn}) has a pole exactly if
$(\zeta_b^a\zeta_k^{hA};\zeta_k^{hB})_{n+1}(\zeta_b^{-a}\zeta_k^{h(B-A)};\zeta_k^{hB})_{n+1}=0$ for some $n\in\N_0$.
By an elementary calculation, we have the following characterization, where $\lfloor x\rfloor$ is the floor of $x$, $k':=k/(k,B)$, and $\alpha_n$ denotes the order of the zero in $(\zeta_b^a\zeta_k^{hA};\zeta_k^{hB})_{n+1}(\zeta_b^{-a}\zeta_k^{h(B-A)};\zeta_k^{hB})_{n+1}$.
\begin{lemma}
\label{poles}
We have that $(\zeta_b^a\zeta_k^{hA};\zeta_k^{hB})_{\infty}(\zeta_b^{-a}\zeta_k^{h(B-A)};\zeta_k^{hB})_{\infty}=0$ if and only if $b|k$ and $(B,k)|(\frac{ak}b+hA)$. Moreover, if  this is the case, then for any $n\in\N_0$, we have $\alpha_n\geq\left\lfloor\frac{2n+2}{k'}\right\rfloor$.
\end{lemma}
Although we have determined exactly when the denominator has a zero, there are two points which must be addressed. Firstly, there could be zeros in the numerator canceling out a pole. We show that this cannot happen in the next subsection by explicitly analyzing the zeros in the numerator. Secondly, one could also imagine that the poles above could somehow cancel each other out. This case, however, is easy to exclude as the orders of the zeros in the denominator are increasing.

\subsection{Zeros in the numerator of $g_2$}
In order to finish our analysis of the poles arising from the denominator of $g_2(\zeta_b^aq^A;q^B)$, we need to know more about the location of zeros in the numerator since these could potentially cancel out the zeros in the denominator. The analysis in this section shows that this never happens. That is, below we prove the following lemma, where 

\[\mathcal Q_{a,b,A,B}:= \left\{\frac hk\colon b|k,\, (B,k)\Big|\left(\frac{ak}b+hA\right)\right\}.\]
\begin{proposition}
\label{zerospolescombined}
Assuming the notation above, the poles in $g_2(\zeta_b^aq^A;q^B)$ arising from zeros in the denominator are exactly at the cusps $\frac hk\in\mathcal Q_{a,b,A,B}$.
\end{proposition}
To show Proposition \ref{zerospolescombined}, we study of the zeros in the numerator of $g_2(\zeta_b^aq^A;q^B)$. After specializing (\ref{g2defn}), we find that the numerator of the $n$-th term of the series defining $g_2(\zeta_b^a\zeta_k^{hA};\zeta_k^{hB})$ is equal (up to a power of $q$) to $(-\zeta_k^{hB};\zeta_k^{hB})_{n}$. 
Another straightforward calculation yields the following lemma, where $\operatorname{ord}_2(m)$ is the $2$-order of $m$ and $\beta_n$ denotes the order of the zero in $(-\zeta_k^{hB};\zeta_k^{hB})_{n}$.
\begin{lemma}\label{zeros}
We have $(-\zeta_k^{hB};\zeta_k^{hB})_{\infty}=0$ if and only if $\operatorname{ord}_2(k)>\operatorname{ord}_2(B)$. Moreover, if $(-\zeta_k^{hB};\zeta_k^{hB})_{\infty}=0$, then $\beta_n=\left\lfloor\frac{2n}{k'}\right\rfloor-\left\lfloor\frac n{k'}\right\rfloor.$
\end{lemma}
\noindent We now prove the main result of this subsection.
\begin{proof}[Proof of Proposition \ref{zerospolescombined}]
Assuming the notation of Lemma \ref{poles} and Lemma \ref{zeros}, it is enough to show that $\alpha_n-\beta_n>0$ for $n\gg0$. In fact, from Lemma \ref{poles} and Lemma \ref{zeros}, we have 
that $\alpha_n-\beta_n\geq\lfloor\frac{2n+2}{k'}\rfloor-\lfloor\frac{2n}{k'}\rfloor+\lfloor\frac n{k'}\rfloor\geq\lfloor\frac{n}{k'}\rfloor$, which is clearly sufficient.
\end{proof}

 \section{Preliminary analysis of Mortenson's formula}
We now take a closer look at the formula of Mortenson in Lemma \ref{MortensonBilaterallemma}, analyzing the individual terms for later use. In particular, consider 
\begin{equation*}L(\zeta;q):=\frac 12\sum_{n\geq0}\frac{q^n\left(\zeta^{-1}q,\zeta\right)_n}{(-q)_n},\quad\quad M(\zeta;q):=\frac{-i\zeta^{\frac12}\vartheta(z;\tau)}{2q^{\frac1{24}}\eta(2\tau)}g_3(-\zeta;q),\end{equation*}
where $(a,b;q)_n:=(a,b)_n:=(a)_n(b)_n$.
The following results shows that if $\frac hk\in\mathcal Q_{a,b,A,B}$, then $L$ becomes a finite sum.
\begin{lemma}
\label{TLSum}

If $\frac hk\in\mathcal Q_{a,b,A,B}$, then $L(\zeta_b^a\zeta_k^{hA};\zeta_k^{hB})$ is a finite, terminating sum. In particular, 

\begin{equation*}L\left(\zeta_b^a\zeta_k^{hA};\zeta_k^{hB}\right)=\frac 12\sum_{n=0}^{k'-1}\frac{\zeta_k^{hBn}\left(\zeta_b^a \zeta_k^{hA},\zeta_b^{-a}\zeta_k^{h(B-A)};\zeta_k^{hB}\right)_n}{\left(-\zeta_k^{hB};\zeta_k^{hB}\right)_n}.\end{equation*}
\end{lemma}
 \begin{proof}
To show that $L$ terminates at the claimed point, it clearly suffices to show that $\alpha_{n-1}-\beta_n\geq0$ for all $n\in\N$ and that $\alpha_{n-1}-\beta_n>0$ for all $n\geq k'$. Indeed, by Lemma \ref{poles} and Lemma \ref{zeros}, we have $\alpha_{n-1}-\beta_n\geq\lfloor\frac{n}{k'}\rfloor$, which completes the proof.
\end{proof}
We also need to determine the limiting behavior of $M$ at cusps in $\mathcal Q_{a,b,A,B}$. In particular, we show that, at these cusps, $M$ tends to zero.
\begin{lemma}\label{MixedMockVanish}
If $\frac hk\in\mathcal Q_{a,b,A,B}$, then $\lim_{q\rightarrow\zeta_k^h}M(\zeta_b^aq^A;q^B)=0.$
\end{lemma}
Before proving Lemma \ref{MixedMockVanish}, we require a bound on the Appell-Lerch sum 
\begin{equation*}A(\zeta;q):=\sum_{n\in\Z}\frac{(-1)^nq^{\frac{3n(n+1)}2}}{1+\zeta q^n}.\end{equation*}
\begin{lemma}
\label{APolyBound}
For any $\frac hk\in\Q$, we have 
$A(\zeta_b^a\zeta_k^{hA}e^{-\frac{At}B};\zeta_k^{hB}e^{-t})\ll t^{-\frac32}\text{ as }t\rightarrow0^+.$
\end{lemma}
\begin{proof}
We begin by estimating
\begin{equation*}\begin{aligned}&\left| A\(\zeta_b^a\zeta_k^{hA}e^{-\frac{At}B};\zeta_k^{hB}e^{-t}\)\right|\leq\sum_{n\in\Z}\frac{e^{-\frac{3tn(n+1)}2}}{\left|1+\zeta_b^a\zeta_k^{h(A+Bn)}e^{-\frac ABt-nt}\right|}
\\ &\leq\frac{1}{\left|1-e^{-\frac {At}{B}}\right|}+\sum_{n\geq1}\frac{e^{-\frac{3tn(n+1)}2}}{\left|1-e^{-\frac ABt-nt}\right|}+\sum_{n\geq1}\frac{e^{-\frac{3tn(n-1)}2}}{\left|1-e^{-\frac ABt+nt}\right|}
\\ &
 \ll\frac 1t+ \sum_{n\geq1}\frac{e^{-\frac{3tn(n+1)}2}}{1-e^{-tn}}+e^{\frac{At}B}\sum_{n\geq1}\frac{e^{-\frac{3tn^2}2}}{1-e^{\frac{At}B-tn}}\ll\frac1t+\frac{e^{\frac{At}{B}}}{1-e^{\frac{At}B-t}}\sum_{n\geq1}e^{-\frac{3n^2t}2}
 \\ &\ll\frac1t+\frac1t\sum_{n\geq1}e^{-\frac{3n^2t}2}.
\end{aligned}\end{equation*}
By comparing with a Gaussian integral, we easily find that the last sum is $\ll t^{-\frac12}$, as desired.

\end{proof}
We are now in a position to prove Lemma \ref{MixedMockVanish}.
\begin{proof}[Proof of Lemma \ref{MixedMockVanish}]
We begin by using Lemma \ref{g2otherappelllemma} to write
\begin{equation}\label{Mformula}M(\zeta;q)=-\frac{i}{2}\zeta^{\frac12}q^{-\frac18}\frac{\vartheta(z;\tau)}{(q)_{\infty}\left(q^2;q^2\right)_{\infty}}A(\zeta;q).\end{equation}
Using Lemma \ref{APolyBound}, it suffices to show that 
\begin{equation*}W(t):=\frac{\vartheta\(\frac ab+\frac{Ah}k+Ait;\frac{hB}{k}+Bit\)}{\eta\(\frac{hB}{k}+Bit\)\eta\(\frac{2hB}{k}+2Bit\)}\end{equation*}
decays exponentially as $t\rightarrow0^+$. For convenience, we will denote the reduced fraction $\frac{2hB}k=:\frac{H}{K}$.
The growth of $W$ may be determined using the modularity properties of $\vartheta$ and $\eta$. 
Specifically, for any $(h,k)=1$, we have the well-known transformation formula (see Chapter 9 of \cite{Rademacher})
\begin{equation*}\eta\(\frac hk+it\)=\sqrt{\frac i{kt}}\chi\(h,[-h]_k,k\)\eta\(\frac {[-h]_k}k+\frac it\),\end{equation*}
where $\chi$ is a certain multiplier and $[h]_k$ is any multiplicative inverse of $h$ modulo $k$, as well as the transformation \cite{Shimura}
\begin{equation*}\vartheta\(z;\frac hk+it\)=\sqrt{\frac i{kt}}\chi\(h,[-h]_k,k\)^3e^{-\frac{\pi z^2}{t}}\vartheta\(\frac{iz}{kt};\frac{[-h]_k}k+\frac it\).\end{equation*}
Hence, we find
\begin{equation}\label{okineed}\begin{aligned}W(t)=&\(-2iKBt\)^{\frac12}\chi\(hB',[-hB']_{k'},k'\)^{2}\chi\(H,[-H]_{K},K\)^{-1}e^{-\frac{\pi}{Bt}\(\frac ab+\frac{Ah}k+Ait\)^2}\\ &\times\frac{\vartheta\(\frac{i}{k'Bt}\(\frac ab+\frac{Ah}k+Ait\);\frac{[-hB']_{k'}}{k'}+\frac{i}{Bt}\)}{\eta\(\frac{[-hB']_{k'}}{k'}+\frac{i}{Bt}\)\eta\(\frac{[-H]_{K}}{K}+\frac{i}{2Bt}\)},\end{aligned}\end{equation}
where $B':=\frac{B}{(k,B)}$.
Using the product expansions of $\eta$ and $\vartheta$, we approximate

\begin{equation}\label{probsneed}\eta\(\frac{[-hB']_{k'}}{k'}+\frac{i}{Bt}\)\sim e^{\frac{\pi i}{12k'}[-hB']_{k'}-\frac {\pi }{12Bt}}\gg e^{-\frac {\pi }{12Bt}},\end{equation}
\begin{equation}
\label{foonew}
\eta\(\frac{[-H]_{K}}{K}+\frac{i}{2Bt}\)\gg e^{-\frac{\pi}{24Bt}},
\end{equation}
and
\begin{equation}\label{referencequestion}\begin{aligned}\vartheta\(\frac{i\(\frac ab+\frac{Ah}k+Ait\)}{k'Bt};\frac{[-hB']_{k'}}{k'}+\frac i{Bt}\)
&\sim -ie^{\frac{\pi i}{4}\(\frac{[-hB']_{k'}}{k'}+\frac{i}{Bt}\)}e^{-\pi i\(\frac{i}{k'Bt}\(\frac ab+\frac{Ah}k+Ait\)\)}
\\ & \ll e^{\frac{\pi}{Bt}\(-\frac14+\frac1{k'}\(\frac ab+\frac{Ah}k\)\)}.
\end{aligned}\end{equation}
\noindent 
Hence, using (\ref{okineed}),(\ref{probsneed}), (\ref{foonew}), and (\ref{referencequestion}), we easily find that
\[W(t)\ll t^{\frac12}e^{\frac{\pi}{Bt}\(-\frac1{8}-\(\frac ab+\frac{Ah}k\)^2+\frac1{k'}\(\frac ab+\frac{Ah}k\)\)}.\]
Thus, it is enough to show that $(\frac ab+\frac{Ah}k)((\frac ab+\frac{Ah}k)-\frac1{k'})\geq0$. As $\frac hk$, $\frac ab$, $k$, and $A$ are non-negative, this inequality holds if $\frac {ak'}b+\frac{Ah}{(k,B)}\geq0.$ 
Hence, the proof is complete once we show that $\frac {ak'}b+\frac{Ah}{(k,B)}\in\N_0$, which is clearly implied by the condition $\frac hk\in\mathcal Q_{a,b,A,B}$. 

\end{proof}
\section{Proof and explicit statement of Theorem \ref{mainthm}}
We are now ready to give the proof of Theorem 1.1, which is split into several propositions. In the first case, we suppose that $g_2(\zeta_b^aq^A;q^B)$ has a root in the denominator. By Proposition \ref{zerospolescombined}, this occurs exactly if $\frac hk\in\mathcal Q_{a,b,A,B}$. In this case, we obtain the following result. 
\begin{proposition}\label{Prop61}
If $\frac hk\in\mathcal Q_{a,b,A,B}$, then
\begin{equation*}\lim_{q\rightarrow\zeta_k^h}\left(g_2\left(\zeta_b^aq^A;q^B\right)-T\left(\zeta_b^aq^A;q^B\right)\right)=-\frac 12\sum_{n=0}^{k'-1}\frac{\zeta_k^{hBn}\left(\zeta_b^a \zeta_k^{hA},\zeta_b^{-a}\zeta_k^{h(B-A)};\zeta_k^{hB}\right)_n}{(-\zeta_k^{hB};\zeta_k^{hB})_n}.\end{equation*}

\end{proposition}

\begin{proof}
Using Lemma \ref{MortensonBilaterallemma} with $\zeta=\zeta_b^aq^A$ and $q\mapsto q^B$, we obtain
\begin{equation*}\begin{aligned}&g_2\(\zeta_b^aq^A;q^B\)+L\(\zeta_b^aq^A;q^B\)=M\(\zeta_b^aq^A;q^B\)+T\(\zeta_b^aq^A;q^B\).\end{aligned}\end{equation*}
By Lemma 4.2, $M(\zeta_b^aq^A;q^B)\rightarrow0$ as $q\rightarrow\zeta_k^h$, and by Lemma 4.1, $L(\zeta_b^aq^A;q^B)$ tends to the negative of the right-hand side of the proposition as $q\rightarrow\zeta_k^h$. This completes the proof.
\end{proof}

We next suppose that $g_2(\zeta_b^aq^A;q^B)$  does not have a root in the denominator. The situation is particularly simple if $\operatorname{ord}_2(k)>\operatorname{ord}_2(B)$, in which case $g_2(\zeta_b^aq^A;q^B)$ becomes a finite sum as $q\rightarrow\zeta_k^h$.

\begin{proposition}\label{Prop62}
 If $\frac hk\not\in\mathcal Q_{a,b,A,B}$ and $\mathrm{ord}_2(k)>\mathrm{ord}_2(B)$, then 
\[\lim_{q\rightarrow\zeta_k^h}g_2\(\zeta_b^aq^A;q^B\)=\sum_{n=0}^{\frac{k'}2-1}\frac{\zeta_k^{\frac{hBn(n+1)}2}\left(-\zeta_k^{hB};\zeta_k^{hB}\right)_n}{\left(\zeta_b^a\zeta_k^{hA},\zeta_b^{-a}\zeta_k^{h(B-A)};\zeta_k^{hB}\right)_{n+1}}.\]
\end{proposition}

\begin{proof}
Using again Lemma 3.2, we see that $g_2(\zeta_b^aq^A;q^B)$ does not have any roots in the denominator as $q\rightarrow\zeta_k^h$. By Abel's Lemma, the limit of $g_2(\zeta_b^aq^A;q^B)$ as $q\rightarrow\zeta_k^h$ equals $g_2(\zeta_b^a\zeta^{hA};\zeta_k^{hB})$, assuming this value exists. We claim that this specialization converges as a terminating series. To see this, note that the $n$-th term in the numerator of $g_2(\zeta_b^aq^A;q^B)$ equals $\zeta_k^{hBn(n+1)/2}(-\zeta_k^{hB};\zeta_k^{hB})_n$, which is zero for $n\geq\frac{k'}2$ by Lemma 3.3. This gives the claim.
\end{proof}

Finally, we consider the case that $\mathrm{ord}_2(k)\leq\mathrm{ord}_2(B)$. In this case, we obtain the following result, where 
\begin{equation*}\begin{aligned}t(\zeta;q)&:=K(z;\tau)-iK\(z+\frac14;\tau+\frac12\)+iq^{-\frac14}\Theta\(2z,\tau,\frac12;2\tau\),\\ 
m(\zeta;q)&:=t\(\zeta_b^aq^A;q^B\)+iT\(i\zeta_b^aq^A;-q^B\),\end{aligned}\end{equation*}
and
\[Q'_{a,b,A,B}:=\left\{\frac hk\in\Q\colon b|2k,\, 2|k,\, (B,k)\Big\vert\(2Ah+\frac{2ak}b+\frac k2\)\right\}.\]
\begin{proposition}\label{propcase2}
Suppose that $\mathrm{ord}_2(k)\leq\mathrm{ord}_2(B)$. Then the following are true, where $k_2$ is the denominator of $\frac hk+\frac1{2B}$, and $k_2':=k_2/(k_2,B)$:

\begin{enumerate}
\item If $\frac hk\in\mathcal Q'_{a,b,A,B}$, then 
\[\lim_{q\rightarrow\zeta_k^h}\left(g_2\(\zeta_b^aq^A;q^B\right)-m\(\zeta_b^aq^A;q^B\)\right)=-\frac i2\sum_{n=0}^{k_2'-1}\frac{\(-\zeta_k^{hB}\)^{n}\left(i\zeta_b^a \zeta_k^{hA},i\zeta_b^{-a}\zeta_k^{h(B-A)};-\zeta_k^{hB}\right)_n}{(\zeta_k^{hB};-\zeta_k^{hB})_n}.\]
\item If $\frac hk\not\in\mathcal Q'_{a,b,A,B}$, then 
\[\lim_{q\rightarrow\zeta_k^h}\left(g_2\left(\zeta_b^aq^A;q^B\right)-t\(\zeta_b^aq^A;q^B\)\right)=i\sum_{n=0}^{\frac{k_2'}2-1}\frac{(-1)^{\frac{n(n+1)}2}\zeta_k^{\frac{hBn(n+1)}2}\left(\zeta_k^{hB};-\zeta_k^{hB}\right)_n}{\left(i\zeta_b^a\zeta_k^{hA},i\zeta_b^{-a}\zeta_k^{h(B-A)};-\zeta_k^{hB}\right)_{n+1}}.\]

\end{enumerate}

\end{proposition}
\begin{proof}
We begin by explaining the idea of the proof. As $q\rightarrow\zeta_k^h$, the Pochhammer symbol in the numerator of $g_2(\zeta_b^aq^A;q^B)$ becomes $(-\zeta_k^{hB};\zeta_k^{hB})_n$. By the assumption $\mathrm{ord}_2(k)\leq\mathrm{ord}_2(B)$, $\zeta_k^{hB}$ is an odd order root of unity. If instead it were an even order root of unity, then the numerator of $g_2$ would have a zero, so that  $g_2(\zeta_b^aq^A;q^B)$  would either be a terminating sum, as in Proposition \ref{Prop62}, or have a root in the denominator as in Proposition \ref{Prop61}. Therefore, if we shift $\frac{hB}k$ by $\frac12$, we obtain the even order root of unity $-\zeta_k^{hB}$ in the second component of $g_2$, which reduces us to a case we have already solved.

To see how to obtain an appropriate shift, we recall Lemma 2.4 and shift the arguments of $g_2$ to find that
\begin{equation}\label{blergh2}\begin{aligned}g_2(i\zeta;-q)
=K\(z+\frac14;\tau+\frac12\)-q^{-\frac14}\mu\left(2z+\frac12,\tau+\frac12;2\tau\right),\end{aligned}\end{equation}
where we used that $\mu(u,v;\tau+1)=\zeta_8^{-1}\mu(u,v;\tau).$ Combining Lemma 2.2, Lemma 2.4, and (\ref{blergh2}) gives
\begin{equation}\label{blergh3}g_2(\zeta;q)=t(\zeta;q)+ig_2(i\zeta;-q).\end{equation}
We next consider $g_2(i\zeta_b^aq^A;-q^B)$, which corresponds to shifting the arguments $\tau\rightarrow\tau+\frac1{2B}$, $\frac ab\rightarrow\frac ab+\frac14-\frac A{2B}$
in $g_2(\zeta_b^aq^A;q^B)$. 
We now determine when $g_2(i\zeta_b^aq^A;-q^B)$ has a zero in the denominator as $q\rightarrow\zeta_k^h$. An elementary calculation shows that whenever $k'$ is odd,
$g_2(i\zeta_b^a\zeta_k^{hA};-\zeta_k^{hB})$ has a zero in the denominator precisely if 
$\frac hk\in\mathcal Q'_{a,b,A,B}.$
 We now finish the proof, distinguishing the two cases corresponding to (i) and (ii). 

(i) By assumption, $k'$ is odd and $\frac hk\in\mathcal Q'_{a,b,A,B}$, and hence $g_2(i\zeta_b^a\zeta_k^{hA};-\zeta_k^{hB})$ has a root in the denominator. Thus, using (\ref{blergh3}) and applying Proposition \ref{Prop61} to the shifted arguments $\frac hk\rightarrow\frac hk+\frac1{2B}$ and $\frac ab\rightarrow\frac ab+\frac14-\frac{A}{2B}$ directly gives the desired formula.

(ii) As $\frac hk\not\in\mathcal Q'_{a,b,A,B}$, $g_2(i\zeta_b^aq^A;-q^B)$ does not have a root in the denominator at $q=\zeta_k^h$. Hence, using (\ref{blergh3}) and applying Proposition \ref{Prop62} with the shifted arguments $\frac hk\rightarrow\frac hk+\frac1{2B}$ and $\frac ab\rightarrow\frac ab+\frac14-\frac{A}{2B}$, gives the right-hand side of (ii).

\end{proof} 
 
\section{Examples and concluding remarks}
\subsection{Examples}
We now illustrate Theorem \ref{mainthm} with an illuminating example. Consider the second order mock theta function $B(q)$ given in (5.1) of \cite{GordonMcIntosh} 
\begin{equation}\label{BDefn}B(q):=\sum_{n\geq0}\frac{q^n\(-q;q^2\)_n}{\(q;q^2\)_{n+1}}.\end{equation}
By a mock theta conjecture (see (5.2) of \cite{GordonMcIntosh}), we have that $B(q)=g_2(q;q^2)$. We now consider the different cases of Theorem \ref{mainthm} for this specialization, which gives the following result, where 
\[N(q):=2q^{-\frac12}\frac{\eta(4\tau)^5}{\eta(2\tau)^4}.\]
\begin{lemma}
\label{lemmaexample}
Assuming the notation above, the following are true:
\begin{enumerate}
\item If $k$ is odd, then \[\lim_{q\rightarrow\zeta_k^h}\(B(q)-N(q)\)=-\frac12\sum_{n=0}^{\frac{k-1}2}\frac{\zeta_k^{2hn}\(\zeta_k^h;\zeta_k^{2h}\)_n^2}{\(-\zeta_k^{2h};\zeta_k^{2h}\)_n}.\]
\item If $4|k$, then \[\lim_{q\rightarrow\zeta_k^h}B(q)=\sum_{n=0}^{\frac k4-1}\frac{\zeta_k^{hn(n+1)}\(-\zeta_k^{2h};\zeta_k^{2h}\)_n}{\(\zeta_k^h;\zeta_k^{2h}\)_{n+1}^2}.\]
\item If $k\equiv2\pmod4$, then \[\lim_{q\rightarrow\zeta_h^k}B(q)=i\sum_{n=0}^{\frac{k}2-1}\frac{(-1)^{\frac{n(n+1)}2}\zeta_k^{hn(n+1)}\left(\zeta_k^{2h};-\zeta_k^{2h}\right)_n}{\left(i\zeta_k^{h};-\zeta_k^{2h}\right)_{n+1}^2}.\]
\end{enumerate}
\end{lemma}

\begin{remark}
The formula in (i) a simplification of Corollary 5.3 of \cite{Mortenson} (note that there is a typo in the third term of his formula). 
\end{remark}

\begin{proof}

(i) Note that $\mathcal Q_{0,1,1,2}=\{\frac hk\colon k\equiv1\pmod2\}$, so that $g_2(\zeta_b^aq^A;q^B)$ has a zero in the denominator at $\zeta_k^h$. Using Proposition 5.1 directly gives the result, where we use (\ref{ThetaIdent}) together with the identities
\begin{equation}\label{ThetaIdentities}\vartheta(z+\tau)=-e^{-\pi i\tau-2\pi iz}\vartheta(z), \vartheta\(\frac12\)=-2\frac{\eta(2\tau)^2}{\eta(\tau)}, \vartheta\(\tau+\frac12;2\tau\)=-q^{-\frac14}\frac{\eta(2\tau)^5}{\eta(\tau)^2\eta(4\tau)^2}\end{equation}
to write 
\begin{equation*}T(q;q^2)=\frac12N(q)+q^{-\frac12}\(\frac{\eta(4\tau)^{17}}{4\eta(2\tau)^8\eta(8\tau)^8}-\frac{\eta(4\tau)^7\eta(\tau)^4}{4\eta(2\tau)^6\eta(8\tau)^4}\).\end{equation*}
Using the standard fact that the order of vanishing of $\eta(n\tau)$ at the cusp $-\frac dc$ is $\frac1{24n}(n,c)^2$ (see  Proposition 2.1 of \cite{Kohler}), we see that for $k$ odd, 
\[\lim_{q\rightarrow\zeta_k^h}\frac{\eta(4\tau)^7\eta(\tau)^4}{4\eta(2\tau)^6\eta(8\tau)^4}=0.\]
Additionally, using the explicit formulas for expansions at cusps of eta functions in Proposition 2.1 of \cite{Kohler}, we may check that at the same cusps, $\frac{\eta(4\tau)^5}{\eta(2\tau)^4}+4\frac{\eta(8\tau)^8}{\eta(4\tau)^7}$ is cuspidal, and hence
\[\lim_{q\rightarrow\zeta_k^h}\(q^{-\frac12}\frac{\eta(4\tau)^{17}}{4\eta(2\tau)^8\eta(8\tau)^8}\)=\frac12\lim_{q\rightarrow\zeta_k^h}N(q).\]
This yields that
\[\lim_{q\rightarrow\zeta_k^h}T\(q;q^2\)=\lim_{q\rightarrow\zeta_k^h}N(q).\]

\noindent The formula in (i) now follows directly from Proposition 5.1, noting that although Proposition 5.1 only states that the sum in (i) terminates after $n=k-1$, in this case it is easily checked that the sum vanishes after the term $n=\frac{k-1}2$. 

(ii) Since $\operatorname{ord}_2(B)=1$, we have $\operatorname{ord}_2(k)>\operatorname{ord}_2(B)$, and thus Proposition 5.2 gives the desired formula.

(iii) Note that $\mathcal Q'_{0,1,1,2}=\{\frac hk\in\Q\colon 4|k\}$. Before applying Proposition 5.3 in this case, we first analyze the term $t(q;q^2)$. We compute, using (\ref{ThetaIdent}), that
\[\mathcal K(\tau):=K(\tau;2\tau):=\frac{\eta(4\tau)^5}{q^{\frac12}\eta(2\tau)^4}.\]
Note that the piece $K(\tau+\frac14;2\tau+\frac12)$ in $t(q;q^2)$ is just $\mathcal K(\tau+\frac14).$
By computing the order of vanishing of the eta quotient at cusps, we find that $\mathcal K(\tau)$ is cuspidal at $\frac hk$ if $4|k$ and is unbounded otherwise. Thus, if $k\equiv2\pmod 4$, then $\lim_{\tau\rightarrow \frac hk}\mathcal K(\tau)$ does not exist, but $\lim_{\tau\rightarrow \frac hk}\mathcal K(\tau+\frac14)=0$.

We next analyze the last piece in $t(q;q^2)$. By (\ref{ThetaIdent}), (\ref{ThetaIdentities}), and the identity $\vartheta'(0;\tau)=-2\pi\eta(\tau)^3$ we can simplify 
\[\Theta\(2\tau,2\tau,\frac12;4\tau\)=-4i\frac{\eta(8\tau)^8}{\eta(4\tau)^7}.\]
Using this, we compute that $\Theta\(2\tau,2\tau,\frac12;4\tau\)$ has a pole at $\frac hk$ if $8\nmid k$.

We finally claim that $t(q;q^2)$ vanishes at $\frac hk$ if $k\equiv 2\pmod 4$. From the fact that $\lim_{\tau\rightarrow \frac hk}\mathcal K(\tau+\frac14)=0,$
it suffices to show that for $k\equiv2\pmod 4$
\[\lim_{\tau\rightarrow\frac hk}\(\frac{\eta(4\tau)^5}{\eta(2\tau)^4}+4\frac{\eta(8\tau)^8}{\eta(4\tau)^7}\)=0,\]
which is easily verified. Thus, by applying Proposition 5.3 (ii), we find the desired formula.
\end{proof}

We note that in analyzing $B$, we also could use Proposition 5.3 to analyze the situation if $k$ is odd. In fact, in this case, one can check that 
\[\lim_{q\rightarrow\zeta_k^h}t\(q;q^2\)=2\lim_{q\rightarrow\zeta_k^h}N(q),\]
which, by comparison with Lemma 6.1 (i), implies that $t(q;q^2)\sim T(q;q^2)$. Hence, by applying Proposition 5.3 and Lemma 6.1 (i),  we find, for any odd $k$, the identity

\begin{equation}\label{firstcurious} 
-\frac12\sum_{n=0}^{\frac{k-1}2}\frac{\zeta_k^{2hn}\(\zeta_k^h;\zeta_k^{2h}\)_n^2}{\(-\zeta_k^{2h};\zeta_k^{2h}\)_n}=i\sum_{n=0}^{k-1}\frac{(-1)^{\frac{n(n+1)}2}\zeta_k^{hn(n+1)}\(\zeta_k^{2h};-\zeta_k^{2h}\)_n}{\(i\zeta_k^h;-\zeta_k^{2h}\)_{n+1}^2}
.\end{equation}
Moreover, although it is not immediately obvious that $t(q;q^2)$ vanishes at cusps $\frac hk$ with $k\equiv2\pmod 4$, we conclude by giving an alternate explanation of this phenomenon. This is provided by a special $q$-series identity, namely, the mock theta conjecture given in (6.1). Specifically, directly plugging in $q=\zeta_k^h$ yields
\[\lim_{q\rightarrow\zeta_k^h}B(q)=\sum_{n=0}^{\frac{k-2}4}\frac{\zeta_k^{hn}\(-\zeta_k^h;\zeta_k^{2h}\)_n}{\(\zeta_k^h;\zeta_k^{2h}\)_{n+1}}.\]
By comparing with Lemma \ref{lemmaexample} (iii), we find the curious identity 
\begin{equation}\label{secondcurious}\sum_{n=0}^{\frac{k-2}4}\frac{\zeta_k^{hn}\(-\zeta_k^h;\zeta_k^{2h}\)_n}{\(\zeta_k^h;\zeta_k^{2h}\)_{n+1}}=i\sum_{n=0}^{\frac{k}2-1}\frac{(-1)^{\frac{n(n+1)}2}\zeta_k^{hn(n+1)}\left(\zeta_k^{2h};-\zeta_k^{2h}\right)_n}{\left(i\zeta_k^{h};-\zeta_k^{2h}\right)_{n+1}^2}.\end{equation}
We challenge the reader to prove either (\ref{firstcurious}) or (\ref{secondcurious}) directly. 
\subsection{Concluding remarks and questions}
We conclude with a few comments and questions for future work. 
Firstly, we note that although the simplification for $T(q;q^2)$ in the last subsection may seem like a coincidence, a close inspection of the asymptotics in Proposition 5.1 shows that if we assume $\frac hk\in\mathcal Q_{a,b,A,B}$ and further that $k'$ is odd, then we generally find that $T(\zeta_b^aq^A;q^B)\sim\alpha_{h,k}K(\zeta_b^aq^A;q^B)$ as $q\rightarrow\zeta_k^h$ for some constant $\alpha_{h,k}$ depending on $\frac hk$. Although the resulting formula involves less terms, we chose to express Proposition 5.1 as is for two reasons. Firstly, the constants depend on the choice of the root of unity $q$ approaches, and hence one doesn't obtain a finite number of modular forms in Theorem 1.1, as Propositions 5.1, 5.2, and 5.3 provide, and as conjectured by Rhoades. Secondly, this requires a further assumption that $k'$ is odd. Although this is not a serious problem due to the fact that the case of $k'$ even is covered in Proposition 5.2, we prefer to express Proposition 5.1 in a more general form. Due to the simple nature of the resulting function in Proposition 5.1, and considering the second remark on page 4, we are led to our first problem, which we leave for the interested reader. 
\begin{question}
Can one use Proposition 5.1 and (5.2) of \cite{GordonMcIntosh} to reprove Theorem 1.2 of \cite{FolsomOnoRhoades}?
\end{question}
Furthermore, similarly to the proof of Theorem 1.2 of \cite{FolsomOnoRhoades}, it may be possible to give another derivation of the modular forms in Theorem \ref{mainthm}.
\begin{question}
Can one find suitable modular forms in Theorem \ref{mainthm} directly from studying specializations of the $\mu$ function using asymptotic expansions?
\end{question}
Finally, we note that for general mock theta functions, there is another possible method to obtain the constants in Theorem \ref{mainthm} directly from the knowledge of the shadow of $g_2$ without using hypergeometric formulas. This follows by work of Choi, Lim, and Rhoades relating these constants to values of certain quantum modular forms. Moreover, related work of the authors \cite{BringmannRolenConference} shows that the resulting values may be expressed in terms of $L$-values of the shadow, which may in turn be written as simple, finite sums. However, this approach does not shed light on the nature of the modular forms in Theorem \ref{mainthm}, so we chose a direct, $q$-hypergeometric approach here. In order to generalize to more general weight mock theta functions or mock modular forms, we believe that the following question is a fundamental problem meriting further study. 
\begin{question}
Is it possible to find the modular forms Theorem \ref{mainthm} directly from the shadow of the associated mock theta function? If so, may this approach be extended to other weights when we do not generically have $q$-hypergeometric expressions
for mock theta functions?
\end{question}

\end{document}